\documentclass{daj}

\usepackage{amsmath, graphicx,hyperref}
\usepackage{amsthm,latexsym, amsfonts,amssymb,amstext}
\usepackage{fullpage} 
\usepackage{changes}
\usepackage[margin=20pt,
font+=small,labelformat=parens,labelsep=space,
skip=6pt,list=false,hypcap=false
]{subcaption}
\usepackage{wrapfig}

\newtheorem{theorem}{Theorem}

\newtheorem{lemma}[theorem]{Lemma}

\newtheorem{definition}[theorem]{Definition}

\newcommand{\Ex}[2]{\mathop{\mathbb{E}}\displaylimits_{#1}\left
	[ #2 \right ]}
\newcommand{\Expect}[1]{\mathop{\mathbb{E}}\left
	[ #1 \right ]}

\dajAUTHORdetails{%
  title = {Coding for Sunflowers}, %% please capitalize all significant words
  author = {Anup Rao},
  plaintextauthor = {Anup Rao},
  plaintexttitle = {Coding for Sunflowers}, %%  title without math or LaTeX
  runningtitle = {Coding for Sunflowers}, 
  runningauthor = {Anup Rao},
  copyrightauthor = {A. Rao},
  keywords = {sunflower, combinatorics},
}   %%% END \dajAUTHORdetails

\dajEDITORdetails{%
   year={2020},
   number={2},
   received={25 September 2019},   % received date: example: 7 January 2017
   revised={3 January 2020},    % Optional revised date (you may comment it out)
   published={25 February 2020},  % published date
   doi={10.19086/da.11887},       % XXX = number of paper, e.g. da006 for paper#6
}   %%% END \dajEDITORdetails

\begin{document}

\begin{frontmatter}[classification=text]
%% EDITOR: this will force the keywords to appear right after the Abstract.
%%   If the abstract is too long and would force the keywords off the
%%   front page, please comment out % [classification=text] above
%%   This way the keywords will be floated on the bottom of the first page
%%   even though the Abstract spills over to the next page.

%%% AUTHOR: Title goes here.  This line is optional.  You must use it
%%   if title has footnote attached or requires nontrivial typesetting,
%%   e.g., inclusion of linebreaks to force nice layout.
\title{Coding for Sunflowers} %% please capitalize all significant words

%%% AUTHOR:
%%% List all authors. If you wish, place grant acknowledgements in \thanks.
%%% In brackets include a short tag for each author.
\author[arao]{Anup Rao}
%\author[johan]{Johan H{\aa}stad\thanks{Supported by...}}
%\author[laci]{L\'aszl\'o Lov\'asz\thanks{Supported by...}}
%\author[andy]{Andrew Chi-Chih Yao\thanks{Supported by...}}

%%% AUTHOR: Abstract goes here
\begin{abstract}
A sunflower is a family of sets that have the same pairwise intersections. We simplify a recent result of Alweiss, Lovett, Wu and Zhang that gives an upper bound on the size of every family of sets of size $k$ that does not contain a sunflower. We show how to use the converse of Shannon's noiseless coding theorem to give a cleaner proof of a similar bound.
\end{abstract}
\end{frontmatter}

%%% AUTHOR: body of paper starts here
\section{Introduction}
A $p$-\emph{sunflower} is a family of $p$ sets whose pairwise intersections are identical. How large can a family of sets of size $k$ be if the family does not contain a $p$-sunflower? Erd\H{o}s and Rado \cite{ErdosR60} were the first to pose and answer this question. They showed that any family with more than $(p-1)^k \cdot k!$ sets of size $k$ must contain a $p$-sunflower. This fundamental  fact has many applications in mathematics and computer science \cite{ErdosS89, Razborov85,FrandsenMS97,GalM07,GopalanMR13,Rossman14,RamamoorthyR18, LovettZ19,LovettSZ19}.

After nearly $60$ years, the correct answer to this question is still not known. There is a  family  of $(p-1)^k$ sets of size $k$ that does not contain a $p$-sunflower, and Erd\H{o}s and Rado conjectured that their lemma could be improved to show that this is essentially the extremal example.  Recently, Alweiss, Lovett, Wu and Zhang \cite{AlweissLWZ19} made substantial progress towards resolving the conjecture. They showed that $(\log k)^k \cdot (p\log \log k)^{O(k)}$ sets ensure the presence of a $p$-sunflower. Subsequently, Frankston, Kahn, Narayanan and Park \cite{FrankstonKNP19}  improved the counting methods  developed in \cite{AlweissLWZ19} to prove a conjecture of Talagrand \cite{Talagrand10} regarding monotone set systems. 

In this work, we give simpler proofs for  these results. Our proofs rely on an encoding argument  inspired by a similar encoding argument used in \cite{AlweissLWZ19, FrankstonKNP19}. The main novelty is our use of Shannon's noiseless coding theorem \cite{Shannon48, Kraft49} to reason about the efficiency of the encoding, which turns out to avoid complications that show up when using vanilla counting.  We show:

\begin{theorem}\label{theorem:sunflower} There is a universal  constant $\alpha >1$ such that every family of more than $(\alpha  p \log(pk))^k$ sets of size $k$  must contain a $p$-sunflower. 
\end{theorem}

Let $r(p,k)$ denote the quantity $\alpha p \log(pk)$. We say\footnote{A similar concept was first used by Talagrand \cite{Talagrand10}.} that a sequence\footnote{Here we state the results for sequences of sets because some applications require the ability to reason about sequences that may repeat sets.} of sets $S_1,\dotsc, S_\ell \subset [n]$ of size $k$ is $r$-spread if for every non-empty set $Z \subset [n]$,  the number of elements of the sequence  that contain $Z$ is at most $r^{k-|Z|}$. We prove that for an appropriate choice of $\alpha$, the following lemma holds:
\begin{lemma}\label{lemma:main}
	If a sequence of more than $r(p,k)^k$ sets of size $k$ is $r(p,k)$-spread, then the sequence must contain $p$ disjoint sets.
\end{lemma}
As far as we know, it is possible that Lemma~\ref{lemma:main} holds even when $r(p,k) = O(p)$. Such a strengthening of Lemma \ref{lemma:main} would imply the sunflower conjecture of Erd\H{o}s and Rado. Lemma \ref{lemma:main} easily implies Theorem \ref{theorem:sunflower}:   we proceed by induction on $k$. When $k = 1$, the theorem holds, since the family contains $p$ distinct sets of size $1$. For $k >1 $, if the sets are not $r$-spread, then there is a non-empty set $Z$ such that more than $r^{k-|Z|}$ of the sets contain $Z$. By induction, and since $r(p,k)$ can only increase with $k$, the family of sets contains a $p$-sunflower. Otherwise, if the sets are $r$-spread, Lemma \ref{lemma:main} guarantees the presence of a $p$-sunflower. 

It only remains to prove Lemma \ref{lemma:main}. In fact, we prove something much stronger: a small random set is very likely to contain \emph{some} set of an $r$-spread family of sets.
%
%	\begin{figure}[t]
%	\centering
%	\includegraphics[width = 0.4 \textwidth]{sunflowerpic.pdf}
%	\caption{A $3$-sunflower.}
%\end{figure}

\section{Random sets and $r$-spread families}
To prove Lemma \ref{lemma:main}, we need to understand the extent to which a small random set $W \subseteq [n]$ contains some set of a large family of sets of size $k$. To that end, it is convenient to use the following definition: 
\begin{definition}Given $S_1,\dotsc, S_\ell \subseteq [n]$, for $x \in [\ell]$ and $W \subseteq [n]$, let $\chi(x,W)$  be equal to $S_y \setminus W$, where $y \in [\ell]$ is chosen to minimize $|S_y \setminus W|$ among all choices with $S_y \subseteq S_x \cup W$. If there are multiple choices for $y$ that minimize $|S_y \setminus W|$, let $y$ be the smallest one. 
\end{definition}
Observe that the definition makes sense even if $S_1, \dotsc, S_\ell$ are not all distinct. When $U \subseteq W$, we have  $|\chi(x,U)| \geq |\chi(x,W)|$. We always have $\chi(x,W) \subseteq S_x$. Moreover, $\chi(x,W) = \emptyset$ if and only if there is an index $y$ for which $S_y \subseteq W$. Our main technical lemma shows that if a long sequence of sets is $r$-spread, then $|\chi(X,W)|$ is likely to be small for a random $X$ and a random small set $W$:

\begin{lemma}\label{lemma:code}
	There is a universal constant $\beta>1$ such that the following holds. Let $0 <\gamma,\epsilon<1/2$. If $r= r(k,\gamma,\epsilon) = \beta \cdot (1/\gamma)  \cdot \log(k/\epsilon)$,  and  $S_1,\dotsc, S_\ell \subseteq [n]$ is an  $r$-spread sequence of  at least $r^k$ sets of size $k$, $X \in [\ell]$ is  uniformly random, and $W \subseteq [n]$ is a uniformly random set of size at least  $\gamma n$ independent of $X$, then $\Expect{|\chi(X,W)|} < \epsilon.$ In particular, $\Pr_W[\exists y, S_y \subseteq W] > 1- \epsilon$.
\end{lemma}
This lemma is of independent interest --- it is relevant to several applications in theoretical computer science \cite{Rossman14, LovettSZ19}. Before we prove Lemma \ref{lemma:code}, let us see how to use it to prove Lemma \ref{lemma:main}. 
%
%\begin{lemma}\label{lemma:code} There is a universal constant $\beta$ such that the following holds. 
%	If $S_1,\dotsc, S_\ell \subseteq [n]$ is an  $r$-spread sequence of sets of size $k$, with $ \ell = \lceil r^k \rceil$, $X \in [\ell]$ is uniformly random, and $W \subseteq [n]$ is an independent uniformly random set of size $\beta \cdot m \cdot  \lceil n/r \rceil$, then $\Expect{|\chi(X,W)|} \leq k/2^{m}$, where here $m$ is an arbitrary non-negative integer.
%\end{lemma}

\begin{proof}[Proof of Lemma \ref{lemma:main}]
	Set $\gamma = 1/(2p)$, $\epsilon = 1/p$. Then  $r = r(k,\gamma, \epsilon) = r(p,k)$.  Let $W_1, \dotsc, W_p$ be a uniformly random partition of $[n]$ into sets of size at least $\lfloor n/p \rfloor$. So, each set $W_i$ is of size at least $\lfloor n/p \rfloor \geq \gamma n$.  By symmetry and linearity of expectation, we can apply Lemma \ref{lemma:code} to conclude that 
	\begin{align*}
	\Ex{X,W_1, \dotsc, W_p}{|\chi(X,W_1)|+ \dotsb + |\chi(X,W_p)|}  = \Ex{X,W_1}{|\chi(X,W_1)|}+\dotsb+\Ex{X,W_p}{|\chi(X,W_p)|}< \epsilon p = 1.
	\end{align*}
	
	Since $|\chi(X,W_1)|+ \dotsb + |\chi(X,W_p)|$ is a non-negative integer, there must be some fixed partition  $W_1,\dotsc, W_p$ for which $$\Ex{X}{|\chi(X,W_1)|+ \dotsb + |\chi(X,W_p)|} = 0.$$ This can happen only if the sequence contains $p$ disjoint sets. 
\end{proof}

Next, we briefly describe a technical tool from information theory, before turning to prove Lemma \ref{lemma:code}.

%
%\begin{proof}
%Let $m = \lceil \log(k/\epsilon) \rceil$. Then we have $\beta \cdot m  \cdot \lceil n/r \rceil \leq \gamma n$, as long as  $\alpha$ is large enough. Applying Lemma \ref{lemma:code}, we conclude that
%\begin{align*}
%\Ex{X,W}{|\chi(X,W)|}\leq k  /2^{m} \leq \epsilon,
%\end{align*}
%as required. The statement about the probability follows from Markov's inequality.
%\end{proof}

\section{Prefix-free encodings}
A \emph{prefix-free} encoding is a map $E:[t]\rightarrow \{0,1\}^*$ into the set of all binary strings,  such that if $i \neq j$, $E(i)$ is not a prefix of $E(j)$. Another way to view such an encoding is as a map from the set $[t]$ to the vertices of the infinite binary tree. The encoding is prefix-free if $E(i)$ is never an ancestor of $E(j)$ in the tree.

Shannon \cite{Shannon48} proved that one can always find a prefix-free encoding such that the expected length of the encoding of a random variable $X \in [t]$ exceeds the entropy of $X$ by at most $1$. Conversely, every encoding must have average length that is at least as large as the entropy. For our purposes, we only need the converse under the uniform distribution. The proof is short, so we include it here. All logarithms are taken base $2$.

\begin{lemma}\label{lemma:shannon}
	Let $E:[t] \rightarrow \{0,1\}^*$ be any prefix-free encoding, and $\ell_i$ be the length of $E(i)$.  Then $(1/t) \cdot \sum_{i = 1}^t \ell_i \geq \log t$.
\end{lemma}
\begin{proof}
	We have
	\begin{align*}
	\log t - (1/t) \cdot \sum_{i=1}^t \ell_i  
	=(1/t) \cdot \sum_{i=1}^t \log (t \cdot 2^{-\ell_i}) 
	\leq  \log \Big(\sum_{i=1}^t 2^{-\ell_i}\Big),
	\end{align*}
	where the inequality follows from the concavity of the logarithm function. The fact that this last quantity is at most $0$  is known as Kraft's inequality \cite{Kraft49}. Consider picking a uniformly random binary string longer than all the encodings. Because the encodings are prefix-free, the probability that this random string contains the encoding of some element of $[t]$ as a prefix is exactly $\sum_{i =t}^t 2^{-\ell_i}$. So, this number is at most $1$, and the above expression is at most $0$.
\end{proof}

\section{Proof of Lemma \ref{lemma:code}}
Removing sets from the sequence can only increase $\Expect{|\chi(X,W)|}$, so without loss of generality, suppose $\ell = \lceil r^k\rceil$.  We shall prove that there is a constant $\kappa >1$ such that the following holds. For each integer $m$ with $0 \leq m \leq r\gamma/\kappa$, if $W$ is a uniformly random set of size at least $\kappa m n/r$,  then $\Expect{|\chi(X,W)|} \leq k \cdot (2/3)^m$. By the choice of $r(k,\gamma,\epsilon)$, setting $m = \lfloor r\gamma/\kappa \rfloor$, we get that when $W$ is a set of size at least $\gamma n$, $\Expect{|\chi(X,W)|} \leq  k \cdot (2/3)^{\lfloor\alpha \log (k/\epsilon) /\kappa\rfloor} < \epsilon$ for $\alpha>1$ chosen large enough.

We prove that  $\Expect{|\chi(X,W)|} \leq k \cdot (2/3)^m$  by induction on $m$.
When $m = 0$, the bound holds trivially. When $m>0$, sample $W = U \cup V$, where $U,V$ are uniformly random disjoint sets, $|U|= u=\lceil \kappa (m-1)  n/r \rceil$, and $|V| =  v\geq \kappa  n/r -1 \geq \kappa n/(2r)$. Note that we always have $\kappa/2 \leq (rv/n)$. Moreover, for $\alpha$ large enough, the sequence being $r$-spread implies that we must have $n/k >6$. Indeed, otherwise there would be at least $r^k \cdot (k/n) \geq r^k/6 > r^{k-1}$ sets that share some common element. In particular, we must have $n-u-k \geq n/3$, a bound that we use later. 

It is enough to prove that for all fixed choices of $U$, $$\Ex{V,X}{|\chi(X,W)|} \leq (2/3) \cdot   \Ex{X}{|\chi(X,U)|}.$$  So, fix $U$. If $\chi(x,U)$ is empty for any $x$, then we have $\Ex{V,X}{|\chi(X,W)|} =\Ex{X}{|\chi(X,U )|}=0$, so there is nothing to prove. Otherwise, we must have $\Ex{X}{|\chi(X,U)|} \geq 1$, since $ |\chi(x,U)| \geq 1$ for all  $x$. 
The number of possible pairs $(V,X)$ is at least $r^k \cdot \binom{n-u}{v}$. Our bound will follow from using Lemma \ref{lemma:shannon}. We give a prefix-free encoding of $(V,X)$ below. In each step, we bound the length of the encoding in terms of $|\chi(X,U)|$, $|\chi(X,W)|$ and $\log \Big(r^k \cdot \binom{n-u}{v} \Big)$. In fact, we shall give a prefix-free encoding of $(V,X)$ where every pair will be encoded using 
$$\log \Big(r^k \cdot \binom{n-u}{v} \Big) + a \cdot |\chi(X,U)|  - b \cdot |\chi(X,W)|$$ bits,
for some $a,b>0$, with $a/b \leq 2/3$. Applying Lemma \ref{lemma:shannon}, we conclude that the expected length of the encoding must satisfy:
$$\log \Big(r^k \cdot \binom{n-u}{v} \Big) + a \cdot \Ex{X}{|\chi(X,U)|}  - b \cdot \Ex{X,V}{|\chi(X,W)|} \geq \log \Big(r^k \cdot \binom{n-u}{v} \Big),$$ 
and so $$\Ex{X,V}{|\chi(X,W)|} \leq (2/3)  \cdot \Ex{X}{|\chi(X,U)|}.$$

\begin{figure}[t]
	\centering
	\includegraphics[width = 0.5 \textwidth]{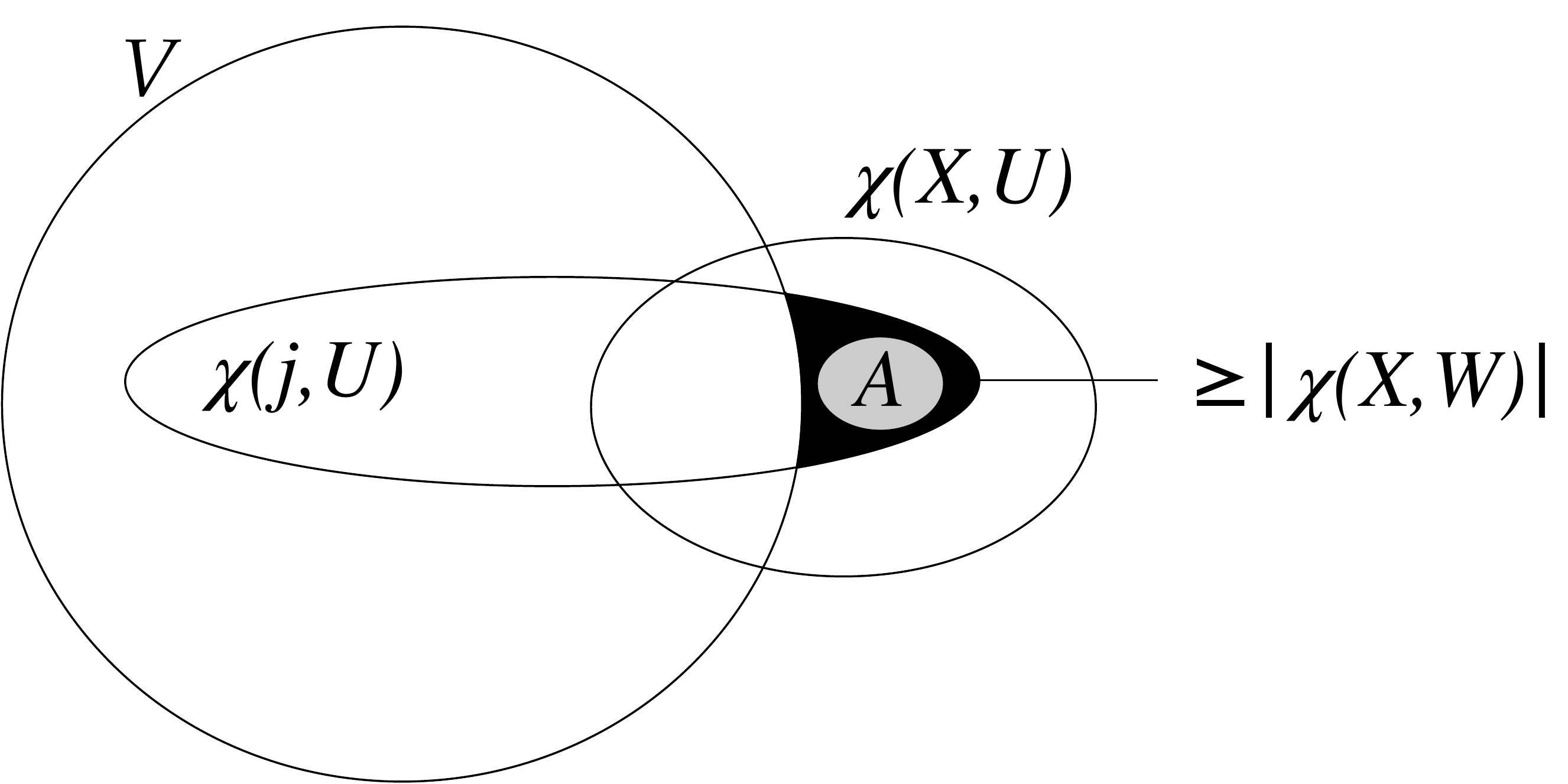}
	\caption{In the first case, given $A$ and $V \cup \chi(X,U)$, the number of candidates for $\chi(X,U)$ is at most $\phi(X,V)$.}
\end{figure}

To describe the encoding, for each $A \subseteq \chi(X,U)$, with  $|A| = |\chi(X,W)|$, define  $$\tau(A,X,V) =	 \{y \in [\ell]: A \subseteq \chi(y,U)\subseteq V \cup \chi(X,U), |\chi(y,U)| = |\chi(X,U)| \},$$ and for $\rho$ a large constant to be set later, define $$\phi(X,V)=r^{k} \cdot  (\rho v/n)^{|\chi(X,U)|} \cdot (vr/n)^{- |\chi(X,W)|}.$$ 

\begin{enumerate}
	\item  The first case is that for all $A$ as above,  $|\tau(A,X,V)| \leq \phi(X,V)$. Then the first bit of the encoding is set to $0$, and we proceed to encode $(V,X)$ like this: 
	\begin{enumerate}
		\item Encode $|\chi(X,U)|$. It suffices to use a trivial   encoding of this integer: we encode it with the string $0^{|\chi(X,U)|}1$, which has length $|\chi(X,U)|+1$.
		\item Encode $W \cup \chi(X,U)$. Since $U$ has been fixed, there are $$\binom{n-u}{v}+ \dotsb + \binom{n-u}{v + |\chi(X,U)|} \leq \binom{n-u+|\chi(X,U)|}{v+|\chi(X,U)|} \leq \binom{n-u}{v} \cdot (n/v)^{|\chi(X,U)|}$$ choices for this set. So, the encoding has length at most $$\log \Big( \binom{n-u}{v}  \cdot  (n/v)^{ |\chi(X,U)|}\Big)  +1.$$ 
		\item Let $j$ be such that $\chi(j, U) \subseteq W \cup \chi(X, U)$, and  $|\chi(j, U)|$ is minimized. If there are multiple choices for $j$ that achieve the minimum, let $j$ be the smallest one. $X$ is a potential candidate for $j$, so we must have $|\chi(j, U)| \leq |\chi(X,U)|$. Encode $\chi(X, U) \cap \chi(j, U)$. Since $j$ is determined, this takes at most  $|\chi(X, U)|$ bits.
		\item We have already encoded $\chi(j, U) \cap \chi(X, U) \subseteq S_X$. We claim that this set must have size at least $|\chi(X,W)|$.  Indeed,  $\chi(j,U) = S_h \setminus U$ for some set $S_h$ of the $r$-spread sequence. We have 
		$$S_h \setminus U = \chi(j,U) \subseteq \chi(X,U) \cup W,$$ so $$S_h \subseteq \chi(X,U) \cup W  \subseteq S_X \cup W.$$ By the definition of $\chi(X,W)$, this implies that  $$|\chi(X,W)| \leq |S_h \setminus W| = |S_h \setminus W \cap \chi(X,U) \setminus W| \leq |\chi(j,U) \cap \chi (X,U)|,$$ as claimed. Let $A$ be the lexicographically first subset of $\chi(j,U) \cap \chi(X,U)$ of size $|\chi(X,W)|$. Now, since $|\tau(A,X,V)| \leq \phi(X,V)$ for all $A$ of size $|\chi(X,W)|$, we can encode $X$ using a binary string of length at most $$\log(\phi(X,V))+1 = \log \Big(r^k \cdot  (\rho v/n)^{|\chi(X,U)|} \cdot (vr/n)^{-|\chi(X,W)|}\Big)+ 1 .$$
		\item Because $X$ has been encoded, $\chi(X,U)$ is also determined. Encode $W \cap \chi(X,U)$. Together with $W \cup \chi(X,U)$, this determines $W$, and so $V$. This last step takes $|\chi(X,U)|$ bits. 
	\end{enumerate}
	Combining all of the above steps, and using the fact that $|\chi(X,U)| \geq 1$, and $vr/n \geq \kappa/2$,  the total length of the encoding in this case is at most 
	\begin{align*}
	&\log \Big(r^k \cdot \binom{n-u}{v} \Big)  +(c + \log(\rho))\cdot |\chi(X,U)|  - \log(\kappa/2) \cdot  |\chi(X,W)|,
	\end{align*}
	where here $c$ is some constant. 
	
	\begin{figure}[t]
		\centering
		\includegraphics[width = 0.4 \textwidth]{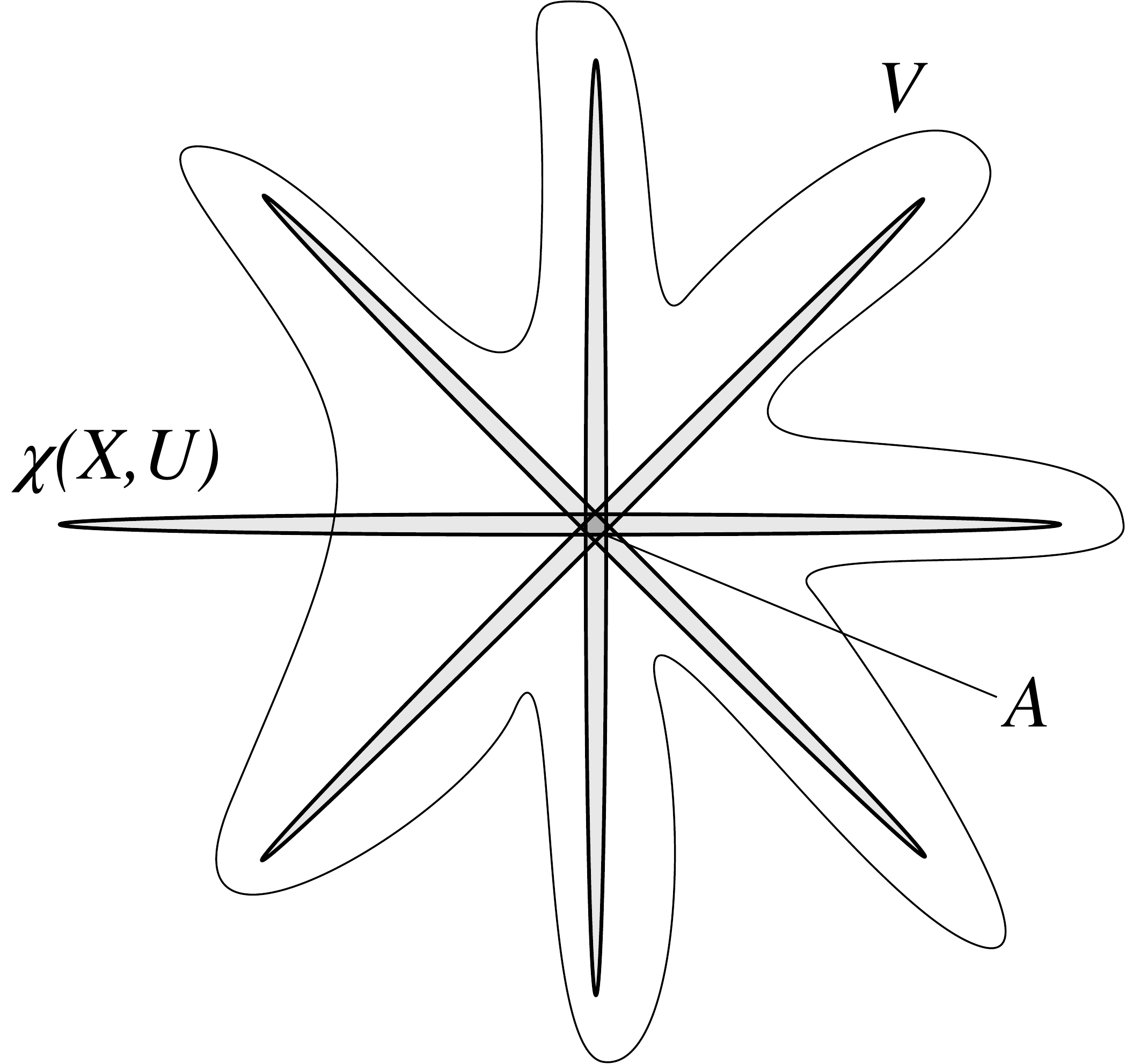}
		\caption{In the second case, given $A$ and $X$, the number of candidates for $V$ is small because $V$ must include an unusually large number of sets of the form $\chi(y,U)\setminus \chi(X,U)$.}
	\end{figure}
	
	\item In the second case, there is a set $A \subseteq \chi(X,U)$ of size $|\chi(X,W)|$ such that  $|\tau(A,X,V)|> \phi(X,V)$.  Then the first bit of the encoding is set to $1$, and we proceed like this:
	\begin{enumerate}
		\item Encode $X$. This takes at most $\log r^k + 1$ bits, since $\ell = \lceil r^k \rceil$.
		\item Now $\chi(X,U)$ is determined. Encode the set $A$ promised above. This takes at most $|\chi(X,U)|$ bits.
		\item Now $\phi(X,V)$ is determined, since $A$ is of size $|\chi(X,W)|$. We claim that the previous steps have reduced the number of candidates for $V$ to at most $\binom{n-u}{v} \cdot (6/\rho)^{|\chi(X,U)|}$. Indeed, consider the following random experiment. Choose a set $B$ uniformly at random from the collection of sets satisfying $A \subseteq B \subseteq \chi(X,U)$, and then sample $V \subseteq [n] \setminus U$ uniformly at random. Consider the collection of $y \in \tau(A,X,V)$ for which $B = \chi(y,U) \cap \chi(X,U)$. Define $$N(A,B,X,V) = |\{y \in [\ell]: B=\chi(y,U) \cap \chi(X,U), y \in \tau(A,X,V)\}|.$$ We have that for the fixed value of $A,X$ specified previously,
		\begin{align*}
		&\Ex{B,V}{N(A,B,X,V)} 
		\leq \Ex{B}{r^{k-|B|} \cdot \Big(\frac{v}{n-u-k}\Big)^{|\chi(X,U)| -|B|}}.
		\end{align*}
		This is because the sequence of sets is $r$-spread, so there are at most $r^{k-|B|}$ sets of the form $\chi(y,U)$ and of size $|\chi(X,U)|$ that intersect $\chi(X,U)$ in  $B$. For each such set, $V$ includes $\chi(y,U) \setminus \chi(X,U) = \chi(y,U) \setminus B$ with probability at most $(v/(n-u-k))^{|\chi(X,U)| -|B|}$. By the choice of $u$, we have $n-u-k \geq n/3$. So, we continue to bound:
		\begin{align*}
		&\leq \Ex{B}{r^{k-|B|} \cdot \Big(\frac{3v}{n}\Big)^{|\chi(X,U)| -|B|}} \leq r^{k} \cdot (3v/n)^{|\chi(X,U)|} \cdot (vr/n)^{-|\chi(X,W)|}.
		\end{align*}
		The last inequality holds because $|B| \geq |\chi(X,W)|$. On the other hand, we have $$\Pr_V[|\tau(A,X,V)| > \phi(X,V)] \cdot 2^{-|\chi(X,U)|} \cdot \phi(X,V) \leq \Ex{B,V}{N(A,B,X,V)},$$
		since $B$ takes each value with probability at least $2^{-|\chi(X,U)|}$.  By the definition of $\phi(X,V)$, this last inequality can be rewritten as $$\Pr_V[|\tau(A,X,V)|>\phi(X,V)] \leq  (6/\rho)^{|\chi(X,U)|}.$$ So, we can encode $V$ at a cost of  $$\log \binom{n-u}{v} +  \log(6/\rho) \cdot |\chi(X,U)|+1.$$ 
	\end{enumerate}
	Thus, for some constant $c'$, the cost of carrying out the encoding in the second case is at most: 
	\begin{align*}
	&\log \Big(r^k \binom{n-u}{v} \Big) + (\log (1/\rho) + c')\cdot |\chi(X,U)| \\
	&\leq \log \Big(r^k \binom{n-u}{v} \Big) + (\log(\rho)+c) \cdot |\chi(X,U)|  - (2\log(\rho) + c - c')  \cdot |\chi(X,W)|,
	\end{align*}
 where the last inequality was obtained by adding $(2\log \rho +c - c') \cdot (|\chi(X,U)| - |\chi(X,W)|)$, which is non-negative for $\rho$ chosen large enough.
\end{enumerate}
	Set $\rho$ to be large enough so that $(\log(\rho)+c)/(2 \log \rho + c-c') \leq 2/3$,  and $\kappa$ to be large enough so that $\log(\kappa/2) \geq (2 \log (\rho) + c - c')$ to complete the proof. 
\section*{Acknowledgments} %%  you may comment this out if no Ackno
Thanks to Ryan Alweiss, Shachar Lovett, Kewen Wu and Jiapeng Zhang for many useful comments. Thanks to Sivaramakrishnan Natarajan Ramamoorthy, Siddharth Iyer and Paul Beame for useful conversations. Thanks to the editor and reviewers of Discrete Analysis for insightful comments.

%%% AUTHOR:
%%% Bibliography goes here. Note that the arXiv cannot process bibtex
%%% or biber bibliographies.  Example of acceptable bibliograpy format:
\bibliographystyle{amsplain}

%% AUTHOR: You can generate such a bibliography from a .bib file by 
%% running pdflatex/bibtex/pdflatex/pdflatex and then pasting the .bbl file
%% between \begin{thebibliography} and \end{bibliography}

%%% AUTHOR: Include a short description of each author following the
%%% structure below. Use the same short tags used previously.  
%%% Use \imageat{} and \imagedot{} instead of "@" and "." in
%%% email addresses-this replaces the symbols with graphics to avoid 
%%% e-mail address harvesting from the .pdf file
\begin{dajauthors}
\begin{authorinfo}[arao]
  Anup Rao\\
  University of Washington\\
  Seattle, Washington, USA\\
  anuprao\imageat{}cs\imagedot{}washington\imagedot{}edu \\
  \url{https://homes.cs.washington.edu/~anuprao/}
\end{authorinfo}
\end{dajauthors}

\end{document}